\documentclass{article}
\usepackage{amsmath,amsthm,amsopn,amstext,amscd,amsfonts,amssymb}
\usepackage{natbib}

\def \E {\mathop{\rm E}\nolimits}
\def \tr {\mathop{\rm tr}\nolimits}
\def \re {\mathop{\rm Re}\nolimits}
\def \im {\mathop{\rm Im}\nolimits}

\def \Vol {\mathop{\rm Vol}\nolimits}

\def \etr {\mathop{\rm etr}\nolimits}
\def \diag {\mathop{\rm diag}\nolimits}
\def \build#1#2#3{\mathrel{\mathop{#1}\limits^{#2}_{#3}}}

\renewenvironment{abstract}
                 {\vspace{6pt}
                  \begin{center}
                  \begin{minipage}{5in}
                  \centerline{\textbf{Abstract}}
                  \noindent\ignorespaces
                 }
                 {\end{minipage}\end{center}}
\newtheorem{thm}{\textbf{Theorem}}[section]
\newtheorem{cor}{\textbf{Corollary}}[section]
\newtheorem{lem}{\textbf{Lemma}}[section]
\newtheorem{prop}{\textbf{Proposition}}[section]
\theoremstyle{definition}
\newtheorem{defn}{\textbf{Definition}}[section]

\newtheorem{rem}{\textbf{Remark}}[section]

\title{\Large \textbf{Distribution theory of quadratic forms with matrix argument}}
\author{
  \textbf{Jos\'e A. D\'{\i}az-Garc\'{\i}a} \thanks{Corresponding author\newline
    }\\
  {\normalsize Department of Statistics and Computation} \\
  {\normalsize 25350 Buenavista, Saltillo, Coahuila, Mexico} \\
  {\normalsize E-mail: jadiaz@uaaan.mx} \\
}
\date{}
\begin{document}
\maketitle

\begin{abstract}
This paper proposes the density and characteristic functions of a general matrix quadratic form
$\mathbf{X}^{*}\mathbf{AX}$, when $\mathbf{A} = \mathbf{A}^{*}$, $\mathbf{X}$ has a matrix
multivariate elliptical distribution and $\mathbf{X}^{*}$ denotes the usual conjugate transpose
of $\mathbf{X}$. These results are obtained for real normed division algebras. With particular
cases we obtained the density and characteristic functions of matrix quadratic forms for matrix
multivariate normal, Pearson type VII, $t$ and Cauchy distributions.
\end{abstract}

\section{Introduction}\label{sec1}

The distribution of a matrix quadratic form in multivariate normal sample has been studied by
diverse authors using zonal, Laguerre and Hayakawa polynomials with matrix argument, in real
and complex cases, see \citet{k:66}, \citet{h:66} and \citet{s:70}, among others. It is
important to observe that these results have been obtained for positive definite matrix
quadratic forms.

In diverse cases, certain statistics are functions of a quadratic form or special types of it,
and play a very important role in classical multivariate statistical analysis.

By replacing the matrix multivariate normal distribution with a matrix multivariate elliptical
Distribution, in classical multivariate analysis one obtains what is now termed generalised
multivariate analysis. By analogy, the distribution of a matrix quadratic form in a matrix
multivariate elliptical sample plays an equally important role in generalised multivariate
analysis, as it is now interesting to study the distribution of a quadratic form assuming a
matrix multivariate elliptical distribution. Some results in this context have been obtained
for particular matrix quadratic forms, see \citet[Chapter II, pp. 137-200]{fa:90}.

In general, many results first described in statistical theory are then found in real case, and
the version for complex case is subsequently studied. In terms of certain concepts and results
derived from abstract algebra, it is possible to propose a unified means of addressing not only
real and complex cases but also quaternion and octonion cases.

In this paper we obtain the density and characteristic functions of a matrix quadratic form of
a matrix multivariate elliptical distribution for real normed division algebras. Furthermore,
these results are obtained when the matrix of the quadratic form is not necessarily positive
definite, which generalises most of the results presented in the literature in this context.
This paper is structured as follows: Section \ref{sec2} provides some definitions and notation
on real normed division algebras, introducing the corresponding matrix multivariate elliptical
distributions. Some results for Jacobians are proposed and two are obtained together with an
extension of one of the basic properties of zonal polynomials. This is also valid for Jack
polynomials for real normed division algebras, termed spherical functions for symmetric cones,
and are also obtained. Section \ref{sec3} then derives the matrix quadratic form density
function, and as corollaries, some results for particular elliptical distributions are
obtained. In Section \ref{sec4}, the characteristic function of a matrix quadratic form is
obtained and some particular cases are studied.

\section{Preliminary results}\label{sec2}

Let us introduce some notation and useful results.

\subsection{Notation and real normed division algebras}\label{sec21}

A comprehensive discussion of real normed division algebras can be found in \citet{b:02}. For
convenience, we shall introduce some notations, although in general we adhere to standard
notations.

For our purposes, a \textbf{vector space} is always a finite-dimensional module over the field
of real numbers. An \textbf{algebra} $\mathfrak{F}$ is a vector space that is equipped with a
bilinear map $m: \mathfrak{F} \times \mathfrak{F} \rightarrow \mathfrak{F}$ termed
\emph{multiplication} and a nonzero element $1 \in \mathfrak{F}$ termed the \emph{unit} such
that $m(1,a) = m(a,1) = 1$. As usual, we abbreviate $m(a,b) = ab$ as $ab$. We do not assume
$\mathfrak{F}$ to be associative. Given an algebra, we freely think of real numbers as elements
of this algebra via the map $\omega \mapsto \omega 1$.

An algebra $\mathfrak{F}$ is a \textbf{division algebra} if given $a, b \in \mathfrak{F}$ with
$ab=0$, then either $a=0$ or $b=0$. Equivalently, $\mathfrak{F}$ is a division algebra if the
operation of left and right multiplications by any nonzero element is invertible. A
\textbf{normed division algebra} is an algebra $\mathfrak{F}$ that is also a normed vector
space with $||ab|| = ||a||||b||$. This implies that $\mathfrak{F}$ is a division algebra and
that $||1|| = 1$.

There are exactly four normed division algebras: real numbers ($\Re$), complex numbers
($\mathfrak{C}$), quaternions ($\mathfrak{H}$) and octonions ($\mathfrak{O}$), see
\citet{b:02}. We take into account that $\Re$, $\mathfrak{C}$, $\mathfrak{H}$ and
$\mathfrak{O}$ are the only normed division algebras; moreover, they are the only alternative
division algebras, and all division algebras have a real dimension of $1, 2, 4$ or $8$, which
is denoted by $\beta$, see \citet[Theorems 1, 2 and 3]{b:02}. In other branches of mathematics,
the parameter $\alpha = 2/\beta$ is used, see \citet{er:05}.

Let $\mathfrak{L}^{\beta}_{m,n}$ be the linear space of all $n \times m$ matrices of rank $m
\leq n$ over $\mathfrak{F}$ with $m$ distinct positive singular values, where $\mathfrak{F}$
denotes a \emph{real finite-dimensional normed division algebra}. Let $\mathfrak{F}^{n \times
m}$ be the set of all $n \times m$ matrices over $\mathfrak{F}$. The dimension of
$\mathfrak{F}^{n \times m}$ over $\Re$ is $\beta mn$. Let $\mathbf{A} \in \mathfrak{F}^{n
\times m}$, then $\mathbf{A}^{*} = \overline{\mathbf{A}}^{T}$ denotes the usual conjugate
transpose.

The set of matrices $\mathbf{H}_{1} \in \mathfrak{F}^{n \times m}$ such that
$\mathbf{H}_{1}^{*}\mathbf{H}_{1} = \mathbf{I}_{m}$ is a manifold denoted ${\mathcal
V}_{m,n}^{\beta}$, termed the \emph{Stiefel manifold} ($\mathbf{H}_{1}$, also known as
\emph{semi-orthogonal} ($\beta = 1$), \emph{semi-unitary} ($\beta = 2$), \emph{semi-symplectic}
($\beta = 4$) and \emph{semi-exceptional type} ($\beta = 8$) matrices, see \citet{dm:99}). The
dimension of $\mathcal{V}_{m,n}^{\beta}$ over $\Re$ is $[\beta mn - m(m-1)\beta/2 -m]$. In
particular, ${\mathcal V}_{m,m}^{\beta}$ with a dimension over $\Re$, $[m(m+1)\beta/2 - m]$, is
the maximal compact subgroup $\mathfrak{U}^{\beta}(m)$ of ${\mathcal L}^{\beta}_{m,m}$ and
consists of all matrices $\mathbf{H} \in \mathfrak{F}^{m \times m}$ such that
$\mathbf{H}^{*}\mathbf{H} = \mathbf{I}_{m}$. Therefore, $\mathfrak{U}^{\beta}(m)$ is the
\emph{real orthogonal group} $\mathcal{O}(m)$ ($\beta = 1$), the \emph{unitary group}
$\mathcal{U}(m)$ ($\beta = 2$), \emph{compact symplectic group} $\mathcal{S}p(m)$ ($\beta = 4$)
or \emph{exceptional type matrices} $\mathcal{O}o(m)$ ($\beta = 8$), for $\mathfrak{F} = \Re$,
$\mathfrak{C}$, $\mathfrak{H}$ or $\mathfrak{O}$, respectively.

We denote by ${\mathfrak S}_{m}^{\beta}$ the real vector space of all $\mathbf{S} \in
\mathfrak{F}^{m \times m}$ such that $\mathbf{S} = \mathbf{S}^{*}$. Let
$\mathfrak{P}_{m}^{\beta}$ be the \emph{cone of positive definite matrices} $\mathbf{S} \in
\mathfrak{F}^{m \times m}$; then $\mathfrak{P}_{m}^{\beta}$ is an open subset of ${\mathfrak
S}_{m}^{\beta}$. Over $\Re$, ${\mathfrak S}_{m}^{\beta}$ consist of \emph{symmetric} matrices;
over $\mathfrak{C}$, \emph{Hermitian} matrices; over $\mathfrak{H}$, \emph{quaternionic
Hermitian} matrices (also termed \emph{self-dual matrices}) and over $\mathfrak{O}$,
\emph{octonionic Hermitian} matrices. Generically, the elements of $\mathfrak{S}_{m}^{\beta}$
shall be termed
 \textbf{Hermitian matrices}, irrespective of the nature of $\mathfrak{F}$. The
dimension of $\mathfrak{S}_{m}^{\beta}$ over $\Re$ is $[m(m-1)\beta+2]/2$.

Let $\mathfrak{D}_{m}^{\beta}$ be the \emph{diagonal subgroup} of $\mathcal{L}_{m,m}^{\beta}$
consisting of all $\mathbf{D} \in \mathfrak{F}^{m \times m}$, $\mathbf{D} = \diag(d_{1},
\dots,d_{m})$ and let $\mathfrak{T}${\tiny U}$_{m,n}^{+\beta}$ be the
\emph{semi-upper-triangular subgroup} of $\mathcal{L}_{m,n}^{\beta}$ consisting of all
$\mathbf{T} \in \mathfrak{F}^{n \times m}$, with $t_{ii} > 0$.

For any matrix $\mathbf{X} \in \mathfrak{F}^{n \times m}$, $d\mathbf{X}$ denotes the\emph{
matrix of differentials} $(dx_{ij})$. Finally, we define the \emph{measure} or volume element
$(d\mathbf{X})$ when $\mathbf{X} \in \mathfrak{F}^{m \times n}, \mathfrak{S}_{m}^{\beta}$,
$\mathfrak{D}_{m}^{\beta}$ or $\mathcal{V}_{m,n}^{\beta}$, see \citet{d:02}.

If $\mathbf{X} \in \mathfrak{F}^{n \times m}$ then $(d\mathbf{X})$ (the Lebesgue measure in
$\mathfrak{F}^{n \times m}$) denotes the exterior product of the $\beta mn$ functionally
independent variables
$$
  (d\mathbf{X}) = \bigwedge_{i = 1}^{n}\bigwedge_{j = 1}^{m}dx_{ij} \quad \mbox{ where }
    \quad dx_{ij} = \bigwedge_{k = 1}^{\beta}dx_{ij}^{(k)}.
$$

If $\mathbf{S} \in \mathfrak{S}_{m}^{\beta}$ (or $\mathbf{S} \in \mathfrak{T}_{L}^{\beta}(m)$)
then $(d\mathbf{S})$ (the Lebesgue measure in $\mathfrak{S}_{m}^{\beta}$ or in
$\mathfrak{T}_{L}^{\beta}(m)$) denotes the exterior product of the $m(m+1)\beta/2$ functionally
independent variables (or denotes the exterior product of the $m(m-1)\beta/2 + n$ functionally
independent variables, if $s_{ii} \in \Re$ for all $i = 1, \dots, m$)
$$
  (d\mathbf{S}) = \left\{
                    \begin{array}{ll}
                      \displaystyle\bigwedge_{i \leq j}^{m}\bigwedge_{k = 1}^{\beta}ds_{ij}^{(k)}, &  \\
                      \displaystyle\bigwedge_{i=1}^{m} ds_{ii}\bigwedge_{i < j}^{m}\bigwedge_{k = 1}^{\beta}ds_{ij}^{(k)}, &
                       \hbox{if } s_{ii} \in \Re.
                    \end{array}
                  \right.
$$
The context generally establishes the conditions on the elements of $\mathbf{S}$, that is, if
$s_{ij} \in \Re$, $\in \mathfrak{C}$, $\in \mathfrak{H}$ or $ \in \mathfrak{O}$. It is
considered that
$$
  (d\mathbf{S}) = \bigwedge_{i \leq j}^{m}\bigwedge_{k = 1}^{\beta}ds_{ij}^{(k)}
   \equiv \bigwedge_{i=1}^{m} ds_{ii}\bigwedge_{i < j}^{m}\bigwedge_{k =
1}^{\beta}ds_{ij}^{(k)}.
$$
Observe, too, that for the Lebesgue measure $(d\mathbf{S})$ defined thus, it is required that
$\mathbf{S} \in \mathfrak{P}_{m}^{\beta}$, that is, $\mathbf{S}$ must be a non singular
Hermitian matrix (Hermitian positive definite matrix).

If $\mathbf{\Lambda} \in \mathfrak{D}_{m}^{\beta}$ then $(d\mathbf{\Lambda})$ (the Legesgue
measure in $\mathfrak{D}_{m}^{\beta}$) denotes the exterior product of the $\beta m$
functionally independent variables
$$
  (d\mathbf{\Lambda}) = \bigwedge_{i = 1}^{n}\bigwedge_{k = 1}^{\beta}d\lambda_{i}^{(k)}.
$$
If $\mathbf{H}_{1} \in \mathcal{V}_{m,n}^{\beta}$ then
$$
  (\mathbf{H}^{*}_{1}d\mathbf{H}_{1}) = \bigwedge_{i=1}^{n} \bigwedge_{j =i+1}^{m}
  \mathbf{h}_{j}^{*}d\mathbf{h}_{i}.
$$
where $\mathbf{H} = (\mathbf{H}_{1}|\mathbf{H}_{2}) = (\mathbf{h}_{1}, \dots,
\mathbf{h}_{m}|\mathbf{h}_{m+1}, \dots, \mathbf{h}_{n}) \in \mathfrak{U}^{\beta}(m)$. It can be
proved that this differential form does not depend on the choice of the $\mathbf{H}_{2}$
matrix. Furthermore, $(\mathbf{H}^{*}_{1}d\mathbf{H}_{1})$ is invariant under the matrix
transformations
$$
  \mathbf{H}_{1} \rightarrow \mathbf{H}_{1}\mathbf{Q}, \quad \mathbf{Q}\in
  \mathfrak{U}^{\beta}(m) \mbox{ and }  \mathbf{H}_{1} \rightarrow \mathbf{P}\mathbf{H}_{1}, \quad \mathbf{P}\in
  \mathfrak{U}^{\beta}(n),
$$
and defines an invariant measure on the Stiefel manifold $\mathcal{V}_{m,n}^{\beta}$, see
\citet{j:54}.

When $m = 1$; $\mathcal{V}^{\beta}_{1,n}$ defines the unit sphere in $\mathfrak{F}^{n}$. This
is, of course, an $(n-1)\beta$- dimensional surface in $\mathfrak{F}^{n}$. When $m = n$ and
denotes $\mathbf{H}_{1}$ by $\mathbf{H}$, $(\mathbf{H}^{*}d\mathbf{H})$ is termed the
\emph{Haar measure} on $\mathfrak{U}^{\beta}(m)$.

The surface area or volume of the Stiefel manifold $\mathcal{V}^{\beta}_{m,n}$ is
\begin{equation}\label{vol}
    \Vol(\mathcal{V}^{\beta}_{m,n}) = \int_{\mathbf{H}_{1} \in
  \mathcal{V}^{\beta}_{m,n}} (\mathbf{H}^{*}_{1}d\mathbf{H}_{1}) =
  \frac{2^{m}\pi^{mn\beta/2}}{\Gamma^{\beta}_{m}[n\beta/2]},
\end{equation}
and therefore
$$
  (d\mathbf{H}_{1}) = \frac{1}{\Vol\left(\mathcal{V}^{\beta}_{m,n}\right)}
    (\mathbf{H}_{1}^{*}d\mathbf{H}_{1}) = \frac{\Gamma^{\beta}_{m}[n\beta/2]}{2^{m}
    \pi^{mn\beta/2}}(\mathbf{H}_{1}^{*}d\mathbf{H}_{1}).
$$
is the \emph{normalised invariant measure on} $\mathcal{V}^{\beta}_{m,n}$ and $(d\mathbf{H})$,
i.e., with  $(m = n)$, it defines the \emph{normalised Haar measure} on
$\mathfrak{U}^{\beta}(m)$. In (\ref{vol}), $\Gamma^{\beta}_{m}[a]$ denotes the multivariate
Gamma function for the space $\mathfrak{S}_{m}^{\beta}$, and is defined by
\begin{eqnarray*}
  \Gamma_{m}^{\beta}[a] &=& \displaystyle\int_{\mathbf{A} \in \mathfrak{P}_{m}^{\beta}}
  \etr\{-\mathbf{A}\} |\mathbf{A}|^{a-(m-1)\beta/2 - 1}(d\mathbf{A}) \\
&=& \pi^{m(m-1)\beta/4}\displaystyle\prod_{i=1}^{m} \Gamma[a-(i-1)\beta/2],
\end{eqnarray*}
where $\etr(\cdot) = \exp(\tr(\cdot))$, $|\cdot|$ denotes the determinant and $\re(a)
> (m-1)\beta/2$, see \citet{gr:87}.

\subsection{Jacobians}\label{sec22}

In this section, we summarise diverse Jacobians with respect to Lebesgue measure in terms of
the $\beta$ parameter. For a detailed discussion of this and related issues, see \citet{d:02},
\citet{er:05}, \citet{f:05} and \citet{k:84}. Two new Jacobians for singular linear
transformations are proposed as well.

\begin{prop}\label{lemlt}
Let $\mathbf{A} \in \mathfrak{L}_{n,n}^{\beta}$, $\mathbf{B} \in \mathfrak{L}_{m,m}^{\beta}$
and $\mathbf{C} \in \mathfrak{L}_{m,n}^{\beta}$ be matrices of constants, $\mathbf{Y}$ and
$\mathbf{X} \in \mathfrak{L}_{m,n}^{\beta}$ matrices of functionally independent variables such
that $\mathbf{Y} = \mathbf{AXB} + \mathbf{C}$. Then
\begin{equation}\label{lt}
    (d\mathbf{Y}) = |\mathbf{A}^{*}\mathbf{A}|^{\beta m/2} |\mathbf{B}^{*}\mathbf{B}|^{\beta
    n/2}(d\mathbf{X}).
\end{equation}
\end{prop}

\begin{prop}[Singular value decomposition, $SVD$]\label{lemsvd}
Let $\mathbf{X} \in \mathfrak{L}_{m,n}^{\beta}$, such that $\mathbf{X} =
\mathbf{V}_{1}\mathbf{DW}^{*}$ with $\mathbf{V}_{1} \in {\mathcal V}_{m,n}^{\beta}$,
$\mathbf{W} \in \mathfrak{U}^{\beta}(m)$ and $\mathbf{D} = \diag(d_{1}, \cdots,d_{m}) \in
\mathfrak{D}_{m}^{1}$, $d_{1}> \cdots > d_{m} > 0$. Then
\begin{equation}\label{svd}
    (d\mathbf{X}) = 2^{-m}\pi^{\tau} \prod_{i = 1}^{m} d_{i}^{\beta(n - m + 1) -1}
    \prod_{i < j}^{m}(d_{i}^{2} - d_{j}^{2})^{\beta} (d\mathbf{D}) (\mathbf{V}_{1}^{*}d\mathbf{V}_{1})
    (\mathbf{W}^{*}d\mathbf{W}),
\end{equation}
where
$$
  \tau = \left\{
             \begin{array}{rl}
               0, & \beta = 1; \\
               -m, & \beta = 2; \\
               -2m, & \beta = 4; \\
               -4m, & \beta = 8.
             \end{array}
           \right.
$$
\end{prop}

\begin{prop}[ Spectral decomposition]\label{lemsd}
Let $\mathbf{S} \in \mathfrak{P}_{m}^{\beta}$. Then the spectral decomposition can be written
as $\mathbf{S} = \mathbf{W}\mathbf{\Lambda W}^{*}$, where $\mathbf{W} \in
\mathfrak{U}^{\beta}(m)$ and $\mathbf{\Lambda} = \diag(\lambda_{1}, \dots, \lambda_{m}) \in
\mathfrak{D}_{m}^{1}$, with $\lambda_{1}> \cdots> \lambda_{m}>0$. Then
\begin{equation}\label{sd}
    (d\mathbf{S}) = 2^{-m} \pi^{\tau} \prod_{i < j}^{m} (\lambda_{i} - \lambda_{j})^{\beta}
    (d\mathbf{\Lambda})(\mathbf{W}^{*}d\mathbf{W}),
\end{equation}
where $\tau$ is defined in Proposition \ref{lemsvd}.
\end{prop}

\begin{prop}\label{lemW}
Let $\mathbf{X} \in \mathfrak{L}_{m,n}^{\beta}$, and  $\mathbf{S} = \mathbf{X}^{*}\mathbf{X}
\in \mathfrak{P}_{m}^{\beta}.$ Then
\begin{equation}\label{w}
    (d\mathbf{X}) = 2^{-m} |\mathbf{S}|^{\beta(n - m + 1)/2 - 1}
    (d\mathbf{S})(\mathbf{V}_{1}^{*}d\mathbf{V}_{1}),
\end{equation}
with $\mathbf{V}_{1} \in {\mathcal V}_{m,n}^{\beta}$.
\end{prop}

Now, let $\mathfrak{L}_{m,n}^{\beta}(q)$ be the linear space of all $n \times m$ matrices of
rank $q \leq \min(n,m)$ with $q$ distinct singular values. In addition, observe that, if
$\mathbf{X} \in \mathfrak{L}_{m,n}^{\beta}(q)$, we can write $\mathbf{X}$ as
$$
  \mathbf{X}_{1} = \left (
          \begin{array}{cc}
            \build{\mathbf{X}_{11}}{}{q \times q} &  \build{\mathbf{X}_{12}}{}{q \times m-q} \\
            \build{\mathbf{X}_{21}}{}{n-q \times q} &  \build{\mathbf{X}_{11}}{}{n-q \times m-q} \\
          \end{array}
          \right )
$$
such that $r(\mathbf{X}_{11}) = q$. This is equivalent to the left and right products of matrix
$\mathbf{X}$ with permutation matrices $\mathbf{\Pi}_{1}$ and $\mathbf{\Pi}_{2}$, see
\citet[section 3.4.1, 1996]{gvl:96}, that is $\mathbf{X}_{1} = \mathbf{\Pi}_{1}
\mathbf{X}\mathbf{\Pi}_{2}$. Note that the exterior product of the elements from the
differential matrix $d\mathbf{X}$ is not affected by the fact that we multiply $\mathbf{X}$
(right and left) by a permutation matrix, that is, $(d\mathbf{X}_{1}) = (d(\mathbf{\Pi}_{1}
\mathbf{X}\mathbf{\Pi}_{2})) = (d\mathbf{X})$, since $\mathbf{\Pi}_{1} \in
\mathfrak{U}^{\beta}(n)$ and $\mathbf{\Pi}_{2} \in \mathfrak{U}^{\beta}(m)$, see \citet[Section
2.1, 1982]{m:82} and \citet{j:54}. Then, without loss of generality, $(d\mathbf{X})$ shall be
defined as the exterior product of the differentials $dx_{ij}$, such that $x_{ij}$ are
mathematically independent. It is important to note that we shall have $(nq+mq -q^{2})\beta$
mathematically independent elements in the matrix $\mathbf{X} \in {\mathcal
L}_{m,n}^{\beta}(q)$, corresponding to the elements of $\mathbf{X}_{11}, \mathbf{X}_{12}$ and
$\mathbf{X}_{21}$. Explicitly,
\begin{equation}\label{X}
    (d\mathbf{X}) \equiv (d\mathbf{X}_{11})\wedge(d\mathbf{X}_{12})\wedge(d\mathbf{X}_{21}) =
            \bigwedge_{i=1}^{n}\bigwedge_{j=1}^{q}\bigwedge_{k = 1}^{\beta}dx_{ij}^{(k)}
            \bigwedge_{i=1}^{q} \bigwedge_{j = q+1}^{m}\bigwedge_{k = 1}^{\beta}dx_{ij}^{(k)}.
\end{equation}
Furthermore,
\begin{equation}\label{inva}
    (d\mathbf{X}) = (\mathbf{P}d\mathbf{X}\mathbf{Q}), \mbox{ for } \mathbf{P} \in
    \mathfrak{U}^{\beta}(n), \mathbf{Q} \in \mathfrak{U}^{\beta}(m)
\end{equation}
Observe that an explicit form for $(d\mathbf{X})$ depends on the factorisation (base and
coordinate set) employed to represent $\mathbf{X}$ or, that is, they depend on the measure
factorisation of $(d\mathbf{X})$. For example, by using the non-singular part of the
decomposition in singular values and the non-singular part of the spectral decomposition for
$\mathbf{X}$, then we can find an explicit form for $(d\mathbf{X})$, which is not unique, see
\citet{k:68}, \citet{dgm:97}, \citet{u:94} and \citet{dg:97}. Alternatively, an explicit form
for $(d\mathbf{X})$ can be found in terms of the QR decompositions.

A singular random matrix $\mathbf{X}$ in  $\mathfrak{L}_{m,n}^{\beta}(q)$ does not have a
density with respect to Lebesgue's measure in $\mathfrak{F}^{n \times m}$, but it does possess
a density on a subspace ${\mathcal M} \subset \mathfrak{F}^{n \times m}$; see \citet{k:68},
\citet[p. 527]{r:73}, \citet{dgm:97}, \citet{u:94} and \citet[p. 297]{c:99}. Formally,
$\mathbf{X}$ has a density with respect to Hausdorff's measure, which coincides with Lebesgue's
measure, when the latter is defined on the subspace ${\mathcal M}$; see \citet[p. 247]{b:86},
\citet{u:94}, \citet{dgm:97}, \citet{dgg:05} and \citet{dggj:06}.

Given $\mathbf{X} \in \mathfrak{L}_{m,n}^{\beta}(q)$, and constant $\mathbf{A} \in
\mathfrak{L}_{n,p}^{\beta}(c)$, and $\mathbf{Y} \in \mathfrak{L}_{m,p}^{\beta}(q)$ with $c \geq
q$. We wish to determine the Jacobian of the transform $\mathbf{Y} = \mathbf{AX}$. Let us first
consider the following case:

\begin{thm}\label{theo31} Let $\mathbf{X} \in \mathfrak{L}_{m,n}^{\beta}(n)$,
with $\mathbf{A} \in \mathfrak{L}_{n,p}^{\beta}(n)$ constant and $\mathbf{Y} \in
\mathfrak{L}_{m,p}^{\beta}(n)$. If $\mathbf{Y} = \mathbf{AX}$, then
\begin{equation}\label{teo31}
    (d\mathbf{Y}) = \prod_{i = 1}^{n} \sigma_{i}(\mathbf{A})^{\beta m}(d\mathbf{X}) =
    \prod_{i = 1}^{n} \lambda_{i}(\mathbf{AA}^{*})^{\beta m/2}(d\mathbf{X})
\end{equation}
where $\lambda_{i}(\mathbf{M})$ and $\sigma_{i}(\mathbf{M})$ are the $i$-th non-null eigenvalue
and singular value of $\mathbf{M}$, respectively.
\end{thm}
\begin{proof} Let $\mathbf{A} = \mathbf{H}_{1}\mathbf{D}_{\mathbf{A}}\mathbf{Q}^{*}$ is the non-singular part
of the SVD of $\mathbf{A}$, where $\mathbf{H}_{1} \in {\mathcal V}_{n,p}^{\beta}$,
$\mathbf{D}_{\mathbf{A}} = \diag(\sigma_{1}(\mathbf{A}), \cdots,\sigma_{n}(\mathbf{A}))$, with
$\sigma_{i}(\mathbf{A})$ the $i$-th singular value of $\mathbf{A}$ and $\mathbf{Q} \in
\mathfrak{U}^{\beta}(n)$. Furthermore, note that rank$(\mathbf{Y})=$ rank $(\mathbf{AX})= n$.
By differentiating $\mathbf{Y} =\mathbf{AX}$, we obtain
$$
  d\mathbf{Y} = \mathbf{A}d\mathbf{X} = \mathbf{H}_{1}\mathbf{D}_{\mathbf{A}}\mathbf{Q}^{*}d\mathbf{X}.
$$
Now, let $\mathbf{H}_{2}$ (a function of $\mathbf{H}_{1}$) be such that $\mathbf{H} =
(\mathbf{H}_{1}\vdots \ \mathbf{H}_{2}) \in \mathfrak{U}^{\beta}(n)$, then
$$
  \mathbf{H}^{*}d\mathbf{Y} = \left( \begin{array}{c}
    \mathbf{H}_{1}^{*} \\
    \mathbf{H}_{2}^{*} \\
  \end{array}\right) \mathbf{H}_{1}\mathbf{D}_\mathbf{A}\mathbf{Q}^{*}d\mathbf{X}
    = \left( \begin{array}{c}
    \mathbf{H}_{1}^{*}\mathbf{H}_{1}\mathbf{D}_\mathbf{A}\mathbf{Q}^{*}d\mathbf{X} \\
    \mathbf{H}_{2}^{*}\mathbf{H}_{1}\mathbf{D}_\mathbf{A}\mathbf{Q}^{*}d\mathbf{X} \\
  \end{array}\right)
    = \left( \begin{array}{c}
    \mathbf{D}_\mathbf{A}\mathbf{Q}^{*}d\mathbf{X} \\
    \mathbf{0} \\
  \end{array}\right)
$$
as $\mathbf{H}_{2}^{*}\mathbf{H}_{1}=0$. From (\ref{inva}), $(\mathbf{H}^{*}d\mathbf{Y}) =
(d\mathbf{Y})$ and that $(\mathbf{Q}^{*}d\mathbf{X}) = (d\mathbf{X})$, then by Proposition
\ref{lemlt}
$$
  (d\mathbf{Y}) = |\mathbf{D}_\mathbf{A}|^{\beta m} (d\mathbf{X})
     = \prod_{i =1}^{n} \sigma_{i}(\mathbf{A})^{\beta m} (d\mathbf{X})
     = \prod_{i =1}^{n} \lambda_{i}(\mathbf{AA}')^{\beta m/2} (d\mathbf{X}). \mbox{\qed}
$$
\end{proof}
\begin{rem}\label{rem0}
Note that, we can consider the QR decomposition instead of the SVD of matrix $\mathbf{A}$ in
Theorem \ref{theo31}. That is $\mathbf{A} = \mathbf{H}_{1}\mathbf{T}$, where $\mathbf{H}_{1}
\in {\mathcal V}_{n,p}^{\beta}$ and $\mathbf{T} \in \mathfrak{T}${\tiny U}$_{n,n}^{+\beta}$,
see \citet{dgrg:11}. Alternatively to (\ref{teo31}) we have that
\begin{equation}\label{teo31b}
  (d\mathbf{Y}) = \prod_{i = 1}^{n} t_{ii}^{\beta m}(d\mathbf{X}).
\end{equation}
The proof is parallel to that given in Theorem \ref{theo31}. Additionally note that, when $n =
p$, (\ref{teo31}) and (\ref{teo31b}) they agree.
\end{rem}

\begin{thm}\label{theo32} Let $\mathbf{X} \in \mathfrak{L}_{m,n}^{\beta}(q)$,
with $\mathbf{A} \in \mathfrak{L}_{n,p}^{\beta}(c)$ constant, and $\mathbf{Y} \in
\mathfrak{L}_{m,p}^{\beta}(q)$, with $\min(p,n) \geq c \geq q$. If $\mathbf{Y} = \mathbf{AX}$,
then
\begin{equation}\label{teo32}
    (d\mathbf{Y}) = \frac{\displaystyle\prod_{i = 1}^{q} \lambda_{i}(\mathbf{ACC}^{*}\mathbf{A}^{*})^{\beta m/2}}
    {\displaystyle\prod_{i = 1}^{q} \lambda_{i}(\mathbf{CC}^{*})^{\beta m/2}} \ (d\mathbf{X})
\end{equation}
where $\mathbf{C} \in \mathfrak{L}_{q,n}^{\beta}(q)$.
\end{thm}
\begin{proof} Let $\mathbf{C} \in \mathfrak{L}_{q,n}^{\beta}(q)$ such that $\mathbf{X} =\mathbf{CZ}$
where $\mathbf{Z} \in \mathfrak{L}_{m,q}^{\beta}(q)$ and  let us denote $\mathbf{R} =
\mathbf{AC}$. Then
$$
  \mathbf{Y} = \mathbf{AX} = \mathbf{ACZ} = \mathbf{RZ}.
$$
Observing that rank $(\mathbf{Y}) =$ rank $(\mathbf{RZ}) =$ rank $(\mathbf{Z}) = q$, from
Theorem \ref{theo31} we have
\begin{equation}\label{inter}
  (d\mathbf{Y}) = \prod_{i = 1}^{q} \lambda_{i}(\mathbf{RR}^{*})^{\beta m/2}(d\mathbf{Z})
       = \prod_{i = 1}^{q} \lambda_{i}(\mathbf{ACC}^{*}\mathbf{A}^{*})^{\beta m/2}(d\mathbf{Z}).
\end{equation}
Now, $\mathbf{X} = \mathbf{BZ}$, again applying Theorem \ref{theo31}, we obtain
$$
  (d\mathbf{X}) = \prod_{i = 1}^{q} \lambda_{i}(\mathbf{CC}^{*})^{\beta m/2}(d\mathbf{Z})
$$
from which, substituting $(d\mathbf{Z}) = \prod_{i = 1}^{q}
\lambda_{i}(\mathbf{CC}^{*})^{-\beta m/2}(d\mathbf{X})$ in (\ref{inter}), we obtain the desired
result.\qed
\end{proof}
\begin{rem}\label{rem2}
Without loss of generality, suppose that $\mathbf{C} \in \mathcal{V}_{q,n}^{\beta}$ such that
$\mathbf{C}^{*}\mathbf{A}^{*}\mathbf{AC} = \diag(\lambda_{1}(\mathbf{A}^{*}\mathbf{A}), \dots,
\lambda_{q}(\mathbf{A}^{*}\mathbf{A}))$, then from Theorem \ref{teo32},
$$
  (d\mathbf{Y}) = \prod_{i = 1}^{q} \sigma_{i}(\mathbf{A})^{\beta m}(d\mathbf{X}) =
    \prod_{i = 1}^{q} \lambda_{i}(\mathbf{AA}^{*})^{\beta m/2}(d\mathbf{X})
$$
\end{rem} .

\subsection{Elliptical distributions}\label{sec23}

Now, the generalised  matrix multivariate elliptical distributions for real normed division
algebras it is introduced in current section, see \citet{dgrg:11}. A comprehensive and
systematic study can be found in \citet{fz:90} and \citet{gv:93}, for real case. Similarly, the
elliptical vector distributions have been discussed by \citet{mdm:06} for complex case.

\begin{defn}\label{elllip}
It is said that the random matrix $\mathbf{Y} \in \mathcal{L}_{m,n}^{\beta}$ has a \emph{matrix
variate elliptical distribution}, denoted as $\mathbf{Y} \sim \mathcal{E}_{n \times
m}^{\beta}(\boldsymbol{\mu},\mathbf{\Theta},\mathbf{\Sigma}, h)$, if its density with respect
to the Lebesgue measure on $\mathfrak{F}^{n \times m}$ is
\begin{equation}\label{v-sd}
    \frac{C^{\beta}(m,n)}{|\mathbf{\Sigma}|^{\beta n/2}|\mathbf{\Theta}|^{\beta m/2}}
    h\left \{ \beta\tr \left [\mathbf{\Sigma}^{-1} (\mathbf{Y} - \boldsymbol{\mu})^{*}
    \mathbf{\Theta}^{-1}(\mathbf{Y} - \boldsymbol{\mu})\right ]\right \}.
\end{equation}
where
\begin{equation}\label{cc}
        C^{\beta}(m,n) = \frac{\Gamma_{1}^{\beta}[\beta mn/2]}{2 \pi^{\beta mn/2}}\left\{\int_{\mathfrak{P}_{1}^{\beta}}
        u^{\beta mn -1}h(\beta u^{2}) du \right\}^{-1}
\end{equation}
and where $\mathbf{\Theta} \in \mathfrak{P}_{n}^{\beta}$, $ \mathbf{\Sigma} \in
\mathfrak{P}_{m}^{\beta}$ and $\boldsymbol{\mu} \in \mathcal{L}_{m,n}^{\beta}$ are constant
matrices.
\end{defn}
Observe that this class of matrix multivariate distributions includes normal, contaminated
normal, Pearson type II and VII, Kotz, Jensen-Logistic, power exponential and Bessel
distributions, among others; these distributions have tails that are more or less weighted,
and/or present a greater or smaller degree of kurtosis than the normal matrix multivariate
distribution. In particular, observe that if in Definition \ref{elllip} it is taken that  $h(u)
= \exp(-u/2)$, from (\ref{cc}) it can be readily seen that $C^{\beta}(m,n) =
(2\pi\beta^{-1})^{-mn\beta/2}$. Hence, the density obtained is
\begin{equation}\label{normal}
    \frac{1}{(2\pi\beta^{-1})^{mn\beta/2}|\mathbf{\Sigma}|^{\beta n/2}|\mathbf{\Theta}|^{\beta m/2}}
    \etr\left \{-\frac{\beta}{2}\tr \left [\mathbf{\Sigma}^{-1} (\mathbf{Y} - \boldsymbol{\mu})^{*}
    \mathbf{\Theta}^{-1}(\mathbf{Y} - \boldsymbol{\mu})\right ]\right \},
\end{equation}
which is termed the matrix multivariate normal distribution for real normed division algebras
and is denoted as $\mathbf{Y} \sim \mathcal{N}_{m \times n}^{\beta}(\boldsymbol{\mu},
\mathbf{\Theta}, \mathbf{\Sigma})$.

Similarly, observe that if in Definition \ref{elllip} it is taken that $ h(u)=
\left(1+u/g\right)^{-s}$, where $s, g \in \Re$, $s, g > 0$, $s > \beta mn/2$; from (\ref{cc})
it can be seen that
$$
  C^{\beta}(m,n)=\frac{\Gamma_{1}^{\beta}[s]}{(\pi g \beta^{-1})^{\beta mn/2}\Gamma_{1}^{\beta}\left[s-\frac{\beta
  mn}{2}\right]}.
$$
Therefore, the density is
\begin{equation}\label{t}
    \frac{(\pi g \beta^{-1})^{-\beta mn/2}\Gamma_{1}^{\beta}[s]}{\Gamma_{1}^{\beta}\left[s-\frac{\beta
    mn}{2}\right]|\mathbf{\Sigma}|^{\beta n/2}|\mathbf{\Theta}|^{\beta m/2}}
    \left\{1+\frac{\beta}{g}\tr \left [\mathbf{\Sigma}^{-1} (\mathbf{Y} - \boldsymbol{\mu})^{*}
    \mathbf{\Theta}^{-1}(\mathbf{Y} - \boldsymbol{\mu})\right ]\right\}^{-s},
\end{equation}
which is termed the matrix multivariate Pearson type VII distribution for real normed division
algebras. Observe that when $s =(\beta mn+g)/2$ in (\ref{t}), $\mathbf{Y}$ is said to have a
matrix multivariate $t$ distribution for real normed division algebras with $g$ degrees of
freedom. And in this case, if $g =1$, then $\mathbf{Y}$ is said to have a matrix multivariate
Cauchy distribution for real normed division algebras.

\subsection{Some results on integration}\label{sec24}

First consider the following result proposed by \citet{cm:76} for real case and extended to
real normed division algebras by \citet{dgrg:11}. Observe that given a function
$f(\mathbf{H})$, $\mathbf{H} \in \mathfrak{U}^{\beta}(n)$ is possible to integrate over the
last $n-m$ ($n \geq m$) columns of $\mathbf{H}$, the first $m$ columns being fixed, and then to
integrate over these $m$ columns, that is
\begin{lem}\label{lemcm}
$$
  \int_{\mathfrak{U}^{\beta}(m)}\hspace{-3mm}f(\mathbf{H}_{1},\mathbf{H}_{2})(\mathbf{H}^{*}d\mathbf{H}) =
    \int_{\mathbf{H}_{1} \in \mathcal{V}_{m,n}^{\beta}} \int_{\mathbf{M} \in \mathfrak{U}^{\beta}(n-m)}
    \hspace{-3mm}f(\mathbf{H}_{1},\mathbf{FM})(\mathbf{M}^{*}d\mathbf{M})(\mathbf{H}_{1}^{*}d\mathbf{H}_{1}),
$$
where $\mathbf{H}=(\mathbf{H}_{1}\vdots\mathbf{H}_{2})$, $\mathbf{H}_{1} \in
\mathcal{V}_{m,n}^{\beta}$ and $\mathbf{F} = \mathbf{F}(\mathbf{H}_{1}) \in
\mathcal{V}_{n-m,n}^{\beta}$ with columns orthogonal to those of $\mathbf{H}_{1}$, this is
$\mathbf{FF}^{*} = \mathbf{I}_{n}-\mathbf{H}_{1}\mathbf{H}_{1}^{*}$.
\end{lem}
Jack polynomials for real normed division algebras are also termed spherical functions of
symmetric cones in the abstract algebra context, see \citet{S:97}. In addition, in the
statistical literature, they are termed real, complex, quaternion and octonion zonal
polynomials, or, generically, \textit{general zonal polynomials}, see \citet{j:64},
\citet{m:82}, \citet{k:84}, \citet{rva:05a} and \citet{lx:09}.

The original version of the follow result was stated for real and complex cases in \citet{j:60}
and \citet{j:64}, respectively, and for quaternion case by \citet{lx:09}, see also
\citet{dg:09}. Next these results are extended under more general conditions. This result was
previously indirectly conjectured by \citet[Remark on p. 194]{h:66}.

\begin{lem}\label{lempj}
Let $\mathbf{X}_{1} \in \mathfrak{S}^{\beta}_{m}$ and $\mathbf{X}_{2} \in
\mathfrak{S}^{\beta}_{m}$ and $\mathbf{H} \in \mathfrak{U}^{\beta}(m)$. Then
\begin{equation}\label{pb1}
    \int_{\mathbf{H} \in \mathfrak{U}^{\beta}(m)} C_{\kappa}^{\beta}(\mathbf{X}_{1}\mathbf{H}\mathbf{X}_{2}
   \mathbf{H}^{*}) (d\mathbf{H}) = \frac{C_{\kappa}^{\beta}(\mathbf{X}_{1})
   C_{\kappa}^{\beta}(\mathbf{X}_{2})}{C_{\kappa}^{\beta}(\mathbf{I}_{r})}
\end{equation}
where $C_{\kappa}^{\beta}(\cdot)$, denotes the Jack polynomials (see \citet{S:97} and
\citet{dg:09}), $r$ = rank of $\mathbf{X}_{2}$ and $(d\mathbf{H})$ is the normalised invariant
measure on $\mathfrak{U}^{\beta}(n)$.
\end{lem}
\begin{proof}
As well as taking into account that this new version proposed holds for real normed division
algebras, in the original result proposed \citet{j:60} and \citet{j:64}, $\mathbf{X}_{2}$ is
assumed to be positive definite, and in our case $\mathbf{X}_{2}$ is assumed positive
semidefinite. Then, the proof is the same as that given by \citet[Theorem 7.2.5, pp.
243-244]{m:82}, simply observing that the differential operator $\Delta_{\mathbf{X}_{2}}^{*}$
is also invariant when $\mathbf{X}_{2}$ is positive semidefinite, see \citet{dgcl:06}. Hence,
we take
$$
  \mathbf{X}_{2} =
  \left(
      \begin{array}{cc}
        \mathbf{I}_{r} & \build{\mathbf{0}}{}{r \times r-m} \\
        \build{\mathbf{0}}{}{r-m \times m} & \build{\mathbf{0}}{}{r-m \times r-m}
      \end{array}
   \right )
$$
where $r$ is the rank of $\mathbf{X}_{2}$, the conclusion is obtained observing that
\begin{equation}\label{ku}
  C_{\kappa}(\mathbf{A}) =
  C_{\kappa}
  \left(
      \begin{array}{cc}
        \mathbf{A} & \mathbf{0} \\
        \mathbf{0} & \mathbf{0}
      \end{array}
   \right ),
\end{equation}
\citet[Equation (3.6), p. 88]{ku:85}, see also \citet[Corollary 7.2.4(i), p. 236]{m:82}. \qed

\end{proof}

\begin{thm}\label{teo1i}
Let $\mathbf{X}_{1} \in \mathfrak{S}^{\beta}_{n}$ and $\mathbf{X}_{2} \in
\mathfrak{S}^{\beta}_{m}$ and $\mathbf{H}=(\mathbf{H}_{1}\vdots \mathbf{H}_{2}) \in
\mathfrak{U}^{\beta}(n)$, with $\mathbf{H}_{1} \in \mathcal{V}_{m,n}$. Then
\begin{equation}\label{pb2}
    \int_{\mathbf{H}_{1} \in \mathcal{V}^{\beta}_{m,n}} C_{\kappa}^{\beta}(\mathbf{X}_{1}\mathbf{H}_{1}\mathbf{X}_{2}
   \mathbf{H}_{1}^{*}) (\mathbf{H}_{1}^{*}d\mathbf{H}_{1}) = \Vol(\mathcal{V}_{m,n}^{\beta})\frac{C_{\kappa}^{\beta}(\mathbf{X}_{1})
   C_{\kappa}^{\beta}(\mathbf{X}_{2})}{C_{\kappa}^{\beta}(\mathbf{I}_{r})}
\end{equation}
where $r$ = rank of $\mathbf{X}_{2}$.
\end{thm}
\begin{proof} First, denote by $J$ the integral in the left side of (\ref{pb2}). Then, for $\mathbf{K}
\in \mathfrak{U}^{\beta}(n-m)$
$$
   J = \int_{\mathbf{H}_{1} \in \mathcal{V}^{\beta}_{m,n}} \int_{\mathbf{K} \in \mathfrak{U}^{\beta}(n-m)}
   C_{\kappa}^{\beta}(\mathbf{X}_{1}\mathbf{H}_{1}\mathbf{X}_{2}\mathbf{H}_{1}^{*})(d\mathbf{K})
   (\mathbf{H}_{1}^{*}d\mathbf{H}_{1})
$$
where $(d\mathbf{K})$ denotes the Haar probability measure on $\mathfrak{U}^{\beta}(n-m)$. This
is by (\ref{vol})
$$
  (d\mathbf{K}) = \frac{1}{\Vol(\mathfrak{U}^{\beta}(n-m))}(\mathbf{K}^{*}d\mathbf{K}),
  \mbox{ from where, }
  \int_{\mathbf{K} \in \mathfrak{U}^{\beta}(n-m)}(d\mathbf{K}) = 1.
$$
Then, for $\mathbf{H} = [\mathbf{H}_{1}\vdots \mathbf{H}_{2}] \in \mathfrak{U}^{\beta}(n)$ and
by Lemma \ref{lemcm}
$$
   J = \frac{1}{\Vol(\mathfrak{U}^{\beta}(n-m))}
   \int_{\mathbf{H} \in \mathfrak{U}^{\beta}(n)} C_{\kappa}^{\beta}(\mathbf{X}_{1}
   \mathbf{H}_{1}\mathbf{X}_{2}\mathbf{H}_{1}^{*})(\mathbf{H}^{*}d\mathbf{H}).
$$
Now, normalising the measure $(\mathbf{H}^{*}d\mathbf{H})$ and noting that
$$
  \frac{\Vol(\mathfrak{U}^{\beta}(n))}{\Vol(\mathfrak{U}^{\beta}(n-m))} = \Vol(\mathcal{V}^{\beta}_{m,n})
$$
we obtain:
$$
   J = \Vol(\mathcal{V}^{\beta}_{m,n})
   \int_{\mathbf{H} \in \mathfrak{U}^{\beta}(n)} C_{\kappa}^{\beta}(\mathbf{X}_{1}
   \mathbf{H}_{1}\mathbf{X}_{2}\mathbf{H}_{1}^{*})(d\mathbf{H}).
$$
Finally, observe that
$$
  C_{\kappa}^{\beta}(\mathbf{X}_{1} \mathbf{H}_{1}\mathbf{X}_{2}\mathbf{H}_{1}^{*}) =
  C_{\kappa}^{\beta} \left(\mathbf{X}_{1}
   (\mathbf{H}_{1}\vdots\mathbf{H}_{2})
   \left(
      \begin{array}{cc}
        \mathbf{X}_{2} & \build{\mathbf{0}}{}{m \times n-m} \\
        \build{\mathbf{0}}{}{n-m \times m} & \build{\mathbf{0}}{}{n-m \times n-m}
      \end{array}
   \right )
   \left(
      \begin{array}{c}
        \mathbf{H}_{1}^{*} \\
        \mathbf{H}_{2}^{*}
      \end{array}
   \right)
   \right),
$$
the desired result is obtained from (\ref{ku}) and Lemma \ref{lempj}.\qed
\end{proof}

A basic integral property is cited below. For this purpose, we utilise the complexification
$\mathfrak{S}_{m}^{\beta, \mathfrak{C}} = \mathfrak{S}_{m}^{\beta} + i
\mathfrak{S}_{m}^{\beta}$ of $\mathfrak{S}_{m}^{\beta}$. That is, $\mathfrak{S}_{m}^{\beta,
\mathfrak{C}}$ consists of all matrices $\mathbf{X} \in (\mathfrak{F^{\mathfrak{C}}})^{m \times
m}$ of the form $\mathbf{Z} = \mathbf{X} + i\mathbf{Y}$, with $\mathbf{X}, \mathbf{Y} \in
\mathfrak{S}_{m}^{\beta}$. We refer to $\mathbf{X} = \re(\mathbf{Z})$ and $\mathbf{Y} =
\im(\mathbf{Z})$ as the \emph{real and imaginary parts} of $\mathbf{Z}$, respectively. The
\emph{generalised right half-plane} $\mathbf{\Phi} = \mathfrak{P}_{m}^{\beta} + i
\mathfrak{S}_{m}^{\beta}$ in $\mathfrak{S}_{m}^{\beta,\mathfrak{C}}$ consists of all
$\mathbf{Z} \in \mathfrak{S}_{m}^{\beta,\mathfrak{C}}$ such that $\re(\mathbf{Z}) \in
\mathfrak{P}_{m}^{\beta}$, see \citet{dg:09}.

\begin{thm}\label{teo2i}
Let $\mathbf{Z} \in \mathbf{\Phi}$ and $\mathbf{U} \in \mathfrak{S}_{m}^{\beta}$. Assume
$\gamma = \int_{z \in \mathfrak{P}_{1}^{\beta}}f(z) z^{am-k-1} dz < \infty$. Then
\begin{eqnarray}\label{runze22}
    \int_{\mathbf{X} \in \mathfrak{P}_{m}^{\beta}} f(\tr \mathbf{XZ}) |\mathbf{X}|^{a -(m-1)
    \beta/2-1} C_{\kappa}^{\beta}\left(\mathbf{X}\mathbf{U} \right)
    (d\mathbf{X}) \hspace{3cm} \nonumber\\\label{jpe2}
    =\displaystyle\frac{[a]_\kappa^{\beta}\Gamma_{m}^{\beta}[a]}{\Gamma[am+k]}
     |\mathbf{Z}|^{-a} C_{\kappa}^{\beta}(\mathbf{UZ}^{-1})\cdot \vartheta,
\end{eqnarray}
where $\vartheta = \int_{z \in \mathfrak{P}_{1}^{\beta}}f(z) z^{am+k-1} dz$, $\re(a)
> (m-1)\beta/2 - k_{m}$ and $\kappa = (k_{1}, \dots, k_{m})$ and $k =
k_{1} + \cdots + k_{m}$, see \citet{dg:09}. And where $[a]_{\kappa}^{\beta}$ denotes the
generalised Pochhammer symbol of weight $\kappa$, defined as
$$
  [a]_{\kappa}^{\beta} = \prod_{i = 1}^{m}(a-(i-1)\beta/2)_{k_{i}} = \frac{\pi^{m(m-1)\beta/4}
    \displaystyle\prod_{i=1}^{m} \Gamma[a + k_{i} -(i-1)\beta/2]}{\Gamma_{m}^{\beta}[a]},
$$
where $\re(a) > (m-1)\beta/2 - k_{m}$ and
$$
  (a)_{i} = a (a+1)\cdots(a+i-1),
$$
is the standard Pochhammer symbol.
\end{thm}

Finally, consider the extension of the Wishart's integral for real normed division algebras.

\begin{thm}\label{teo3i}
Let $\mathbf{Y} \in \mathfrak{L}^{\beta}_{m,n}$. Then
$$
   \int_{\mathbf{Y}^{*}\mathbf{Y}= \mathbf{R}} f(\mathbf{Y}^{*}\mathbf{Y}) d(\mathbf{Y}) =
   \frac{\pi^{\beta mn/2}}{\Gamma_{m}^{\beta}[\beta n/2]}|\mathbf{R}|^{\beta(n - m + 1)/2 - 1}
   f(\mathbf{R}).
$$
\end{thm}
\begin{proof} By the Proposition \ref{lemW}
$$
    (d\mathbf{Y}) = 2^{-m} |\mathbf{R}|^{\beta(n - m + 1)/2 - 1}
    (d\mathbf{R})(\mathbf{V}_{1}^{*}d\mathbf{V}_{1}),
$$
with $\mathbf{V}_{1} \in {\mathcal V}_{m,n}^{\beta}$. Hence
$$
   \int_{\mathbf{Y}^{*}\mathbf{Y}= \mathbf{R}} \hspace{-4mm}f(\mathbf{Y}^{*}\mathbf{Y}) d(\mathbf{Y}) =
   2^{-m}\hspace{-2mm}\int_{\mathbf{R} \in \mathfrak{P}_{m}^{\beta}}\int_{\mathbf{V}_{1} \in {\mathcal V}_{m,n}^{\beta}}
   \hspace{-4mm}|\mathbf{R}|^{\beta(n - m + 1)/2 - 1} f(\mathbf{R})(d\mathbf{R})(\mathbf{V}_{1}^{*}d\mathbf{V}_{1}).
$$
from which the conclusion is reached.\qed
\end{proof}

\section{Density function}\label{sec3}

In this section we find the density function of a general matrix quadratic form of a matrix
multivariate elliptical distribution for real normed division algebras.

\begin{thm}\label{theoqfd}
Assume that $\mathbf{X} \sim \mathcal{E}_{n \times m}^{\beta}(\boldsymbol{0}, \mathbf{\Theta},
\mathbf{\Sigma}, h)$  and define $\mathbf{W} = \mathbf{X}^{*}\mathbf{A}\mathbf{X} \in
\mathfrak{S}_{m}^{\beta}$ of rank $r$, with $\mathbf{A} \in \mathfrak{S}_{n}^{\beta}$ of rank
$r \leq  m \leq n$. Then the density function of $\mathbf{W}$ for a real normed division
algebra, is
\begin{equation}\label{qfdeq}
    \frac{ C^{\beta}(m,n) \pi^{\beta mn/2} |\mathbf{W}|^{\beta(n-m+1)/2-1}}{\Gamma_{m}^{\beta}[\beta n/2]
    |\mathbf{\Sigma}|^{\beta n/2}|\mathbf{\Theta}|^{\beta m/2}|\mathbf{\Lambda}|^{\beta m/2}} \sum_{k=0}^{\infty}
    \frac{h^{(k)}(0)}{k!}
    \frac{C_{\kappa}^{\beta}\left(\mathbf{\Theta}^{-1}\mathbf{A}^{+}\right)C_{\kappa}(\beta \mathbf{\Sigma}^{-1}\mathbf{W})}
    {C_{\kappa}(\mathbf{I}_{r})}
\end{equation}
where $\mathbf{A}^{+}$ is the Moore-Penrose inverse of $\mathbf{A}$; $\mathbf{A} =
\mathbf{P}_{1}\boldsymbol{\Lambda}\mathbf{P}_{1}^{*}$ is the non-singular part of the spectral
decomposition of $\mathbf{A}$, whit $\mathbf{P}_{1} \in \mathcal{V}_{r,m}^{\beta}$ and
$\boldsymbol{\Lambda} = \diag(\lambda_{1}(\mathbf{A}), \dots, \lambda_{r}(\mathbf{A}))$,
$\lambda_{1}(\mathbf{A})> \cdots > \lambda_{r}(\mathbf{A}) > 0$; and $h^{(k)}(\cdot)$ is the
$k$th derivative of $h$.
\end{thm}
\begin{proof}
The density function of $\mathbf{W}$ is
$$
   \frac{C^{\beta}(m,n)}{|\mathbf{\Sigma}|^{\beta n/2}|\mathbf{\Theta}|^{\beta m/2}}
  \int_{\mathbf{X}^{*}\mathbf{A}\mathbf{X} = \mathbf{W}}h\left[\beta\tr \mathbf{\Sigma}^{-1}
   \mathbf{X}^{*} \mathbf{\Theta}^{-1}\mathbf{X}\right] (d\mathbf{X})
$$
Let us now consider the transformation $\mathbf{Y} = \mathbf{A}^{1/2}\mathbf{X}$, such that
$(\mathbf{A}^{1/2})^{2} = \mathbf{A}$, then, $\mathbf{X} = \mathbf{A}^{+1/2}\mathbf{Y}$, where
$\mathbf{A}^{+1/2}$ is the Moore-Penrose inverse of $\mathbf{A}^{1/2}$. Since, by Remark
\ref{rem2}
\begin{eqnarray*}
  (d\mathbf{X}) &=& = \prod_{i = 1}^{r} \lambda_{i}(\mathbf{A}^{+1/2}\mathbf{A}^{+1/2*})^{\beta m/2}(d\mathbf{Y})
                = \prod_{i = 1}^{r} \lambda_{i}(\mathbf{A}^{+})^{\beta m/2}(d\mathbf{Y}) \\
   &=& \prod_{i = 1}^{r} \lambda_{i}(\mathbf{A})^{-\beta m/2}(d\mathbf{Y})
                = |\boldsymbol{\Lambda}|^{-\beta m/2}(d\mathbf{Y})
\end{eqnarray*}
where $r$ is the rank of $\mathbf{A}$; $\mathbf{A} =
\mathbf{P}_{1}\boldsymbol{\Lambda}\mathbf{P}_{1}^{*}$ is the nonsingular part of the spectral
decomposition of $\mathbf{A}$, whit $\mathbf{P}_{1} \in \mathcal{V}_{r,m}^{\beta}$ and
$\boldsymbol{\Lambda} = \diag(\lambda_{1}(\mathbf{A}), \dots, \lambda_{r}(\mathbf{A}))$,
$\lambda_{1}(\mathbf{A})> \cdots > \lambda_{r}(\mathbf{A}) > 0$. Hence,
$$
   \mathbf{c}_{1} \int_{\mathbf{Y}^{*}\mathbf{Y} = \mathbf{W}}h\left[\beta\tr \mathbf{\Sigma}^{-1}
   \mathbf{Y}^{*}\mathbf{A}^{+1/2} \mathbf{\Theta}^{-1}\mathbf{A}^{+1/2}\mathbf{Y}\right]
   (d\mathbf{Y}),
$$
where
$$
  \mathbf{c}_{1} =  \frac{C^{\beta}(m,n)}{|\mathbf{\Sigma}|^{\beta n/2}
  |\mathbf{\Theta}|^{\beta m/2}|\boldsymbol{\Lambda}|^{\beta m/2}}.
$$
Since $\mathbf{A}^{+1/2} \mathbf{\Theta}^{-1}\mathbf{A}^{+1/2} \in \mathfrak{S}_{n}^{\beta}$,
the integral is invariant under the matrix transformation
$$
  \mathbf{A}^{+1/2} \mathbf{\Theta}^{-1}\mathbf{A}^{+1/2} \rightarrow
  \mathbf{H}^{*}\mathbf{A}^{+1/2} \mathbf{\Theta}^{-1}\mathbf{A}^{+1/2}\mathbf{H},
  \quad \mathbf{H} \in \mathfrak{U}^{\beta}(n),
$$
and the integration with respect to $\mathbf{H}$ on $\mathfrak{U}^{\beta}(n)$. Therefore,
\begin{equation}\label{eqfd1}
    \mathbf{c}_{1}\int_{\mathbf{Y}^{*}\mathbf{Y} = \mathbf{W}} \int_{\mathbf{H} \in \mathfrak{U}^{\beta}(n)}
    h\left[\beta\tr \mathbf{\Sigma}^{-1} \mathbf{Y}^{*}\mathbf{H}^{*}\mathbf{A}^{+1/2}
    \mathbf{\Theta}^{-1}\mathbf{A}^{+1/2}\mathbf{H}
    \mathbf{Y}\right](d\mathbf{H})(d\mathbf{Y}),
\end{equation}
where $(d\mathbf{H})$ is the normalised invariant Haar measure. Let us now assume that $h$ can
be expanded in series of power, that is
$$
  h(v) = \sum_{k=0}^{\infty} \frac{h^{(k)}(0) v^{k}}{k!}.
$$
Hence, recalling that
$$
    \sum_{\kappa}C_{\kappa}^{\beta}(\mathbf{X}) = (\tr(\mathbf{X}))^{k},
$$
see \citet{dg:09}, the $\int_{\mathbf{H} \in \mathfrak{U}^{\beta}(n)} h[\cdot](d\mathbf{H})$
expression in (\ref{eqfd1}) is
\begin{eqnarray*}
   &=& \sum_{k=0}^{\infty}\frac{h^{(k)}(0)}{k!} \int_{\mathbf{H} \in \mathfrak{U}^{\beta}(n)}
   \left(\beta\tr \mathbf{\Sigma}^{-1} \mathbf{Y}^{*}\mathbf{H}^{*}\mathbf{A}^{+1/2}
    \mathbf{\Theta}^{-1}\mathbf{A}^{+1/2}\mathbf{H} \mathbf{Y}\right)^{k}(d\mathbf{H}), \\
   &=& \sum_{k=0}^{\infty}\sum_{\kappa}\frac{h^{(k)}(0)}{k!} \int_{\mathbf{H} \in \mathfrak{U}^{\beta}(n)}
   C_{\kappa}\left(\beta\mathbf{\Sigma}^{-1} \mathbf{Y}^{*}\mathbf{H}^{*}\mathbf{A}^{+1/2}
    \mathbf{\Theta}^{-1}\mathbf{A}^{+1/2}\mathbf{H} \mathbf{Y}\right)(d\mathbf{H}),
\end{eqnarray*}
and from Theorem \ref{teo1i},
\begin{eqnarray*}
   &=& \sum_{k=0}^{\infty}\sum_{\kappa}\frac{h^{(k)}(0)}{k!}
   \frac{C_{\kappa}\left(\beta\mathbf{Y}\mathbf{\Sigma}^{-1} \mathbf{Y}^{*}\right)
   C_{\kappa}\left(\mathbf{A}^{+1/2}
   \mathbf{\Theta}^{-1}\mathbf{A}^{+1/2}\right)}{C_{\kappa}\left(I_{r}\right)},\\
   &=& \sum_{k=0}^{\infty}\sum_{\kappa}\frac{h^{(k)}(0)}{k!}
   \frac{C_{\kappa}\left(\beta\mathbf{\Sigma}^{-1} \mathbf{Y}^{*}\mathbf{Y}\right)
   C_{\kappa}\left(\mathbf{\Theta}^{-1}\mathbf{A}^{+}\right)}{C_{\kappa}\left(I_{r}\right)}.
\end{eqnarray*}
Then, substituting in (\ref{eqfd1}), we have that the density function of $\mathbf{W}$ is
$$
    \mathbf{c}_{1} \sum_{k=0}^{\infty}\sum_{\kappa} \frac{h^{(k)}(0)}{k!} \frac{
   C_{\kappa}\left(\mathbf{\Theta}^{-1}\mathbf{A}^{+}\right)}{C_{\kappa}\left(I_{r}\right)}
   \int_{\mathbf{Y}^{*}\mathbf{Y} = \mathbf{W}} C_{\kappa}
    \left(\beta\mathbf{\Sigma}^{-1} \mathbf{Y}^{*}\mathbf{Y}\right)(d\mathbf{Y}).
$$
Finally, the desired result is obtained from Theorem \ref{teo3i}. \qed
\end{proof}

\begin{cor} Assume that $\mathbf{X}$ has a matrix multivariate normal distribution for real normed
division algebras. Let $\mathbf{W} = \mathbf{X}^{*}\mathbf{A}\mathbf{X}$, with $\mathbf{A} \in
\mathfrak{S}_{n}^{\beta}$ of rank $r \leq m \leq n$, $\mathbf{A} =
\mathbf{P}_{1}\boldsymbol{\Lambda}\mathbf{P}_{1}^{*}$ is the non-singular part of the spectral
decomposition of $\mathbf{A}$, whit $\mathbf{P}_{1} \in \mathcal{V}_{r,m}^{\beta}$ and
$\boldsymbol{\Lambda} = \diag(\lambda_{1}(\mathbf{A}), \dots, \lambda_{r}(\mathbf{A}))$,
$\lambda_{1}(\mathbf{A})> \cdots > \lambda_{r}(\mathbf{A}) > 0$. Then the density of
$\mathbf{W}$ for real normed division algebras is
$$
    \frac{|\mathbf{W}|^{\beta(n-m+1)/2-1}}{(2 /\beta)^{\beta mn/2}\Gamma_{m}^{\beta}[\beta n/2]
    |\mathbf{\Sigma}|^{\beta n/2}|\mathbf{\Theta}|^{\beta m/2}|\mathbf{\Lambda}|^{\beta m/2}} \sum_{k=0}^{\infty}
    \frac{C_{\kappa}^{\beta}\left(\mathbf{\Theta}^{-1}\mathbf{A}^{+}\right)C_{\kappa}(-\beta \mathbf{\Sigma}^{-1}\mathbf{W}/2)}
    {k! \ \ C_{\kappa}(\mathbf{I}_{r})}
$$
\end{cor}
\begin{proof}
This follows from (\ref{gweq}) noting that for the normal case, $h(u) = \exp\{- u/2\}$ and
$C^{\beta}(m,n) = (2 \pi/\beta)^{-\beta mn/2}$. \qed
\end{proof}

In real case, this result has been found by \citet{k:66} and \citet{h:66} in terms of zonal
polynomials, by \citet{s:70} in terms of Laguerre polynomials and by \citet{gn:00} in terms of
Hayakawa polynomials. All these results were proposed when $\mathbf{A}$ is definite positive.
In addition, \citet{k:66} obtain these results for complex case too.

\begin{cor}\label{cor:P7}
Suppose that $\mathbf{X}$ has a matrix multivariate Pearson type VII distribution for real
normed division algebras. Let $\mathbf{W} = \mathbf{X}^{*}\mathbf{A}\mathbf{X}$, with
$\mathbf{A} \in \mathfrak{S}_{n}^{\beta}$ of rank $r \leq m \leq n$, $\mathbf{A} =
\mathbf{P}_{1}\boldsymbol{\Lambda}\mathbf{P}_{1}^{*}$ is the nonsingular part of the spectral
decomposition of $\mathbf{A}$, whit $\mathbf{P}_{1} \in \mathcal{V}_{r,m}^{\beta}$ and
$\boldsymbol{\Lambda} = \diag(\lambda_{1}(\mathbf{A}), \dots, \lambda_{r}(\mathbf{A}))$,
$\lambda_{1}(\mathbf{A})> \cdots > \lambda_{r}(\mathbf{A}) > 0$. Then the density of
$\mathbf{W}$ for real normed division algebras is
\begin{equation}\label{gweq}
    \frac{ C^{\beta}(m,n) \pi^{\beta mn/2} |\mathbf{W}|^{\beta(n-m+1)/2-1}}{\Gamma_{m}^{\beta}[\beta n/2]
    |\mathbf{\Sigma}|^{\beta n/2}|\mathbf{\Theta}|^{\beta m/2}|\mathbf{\Lambda}|^{\beta m/2}} \sum_{k=0}^{\infty}
    \frac{(s)_{k}}{k!} \frac{C_{\kappa}^{\beta}\left(\mathbf{\Theta}^{-1}\mathbf{A}^{+}\right)
    C_{\kappa}(\beta \mathbf{\Sigma}^{-1}\mathbf{W}/g)} {C_{\kappa}(\mathbf{I}_{r})}
\end{equation}
where
$$
  C^{\beta}(m,n)=\frac{\Gamma_{1}^{\beta}[s]}{(\pi g \beta^{-1})^{\beta mn/2}\Gamma_{1}^{\beta}\left[s-\beta mn/2\right]},
$$
$(s)_{k} = s (s+1)\cdots(s+k-1)$, is the standard Pochhammer symbol, and $s, g \in \Re$, $s, g
> 0$, and $s > mn/2$.
\end{cor}
Where $s, g \in \Re$, $s, g > 0$, $s > \beta mn/2$.
\begin{proof}
In this case we have
$$
  C^{\beta}(m,n)=\frac{\Gamma_{1}^{\beta}[s]}{(\pi g \beta^{-1})^{\beta mn/2}\Gamma_{1}^{\beta}\left[s-\beta mn/2\right]},
  \quad \mbox{ and } \quad h(u)= \left(1+u/g\right)^{-s},
$$
$s, g \in \Re$, $s, g > 0$, $s > \beta mn/2$, see \citet{dgrg:11}. Then
$$
  h(u)^{(k)}=\frac{(s)_{k}}{g^{k}} \left(1+\frac{u}{g}\right)^{-(s+k)}.
$$
From where the desired result is follows:\qed
\end{proof}

\section{Characteristic function}\label{sec4}

In this section we derived the characteristic function of a general matrix quadratic form of a
matrix multivariate elliptical distribution for real normed division algebras.

\begin{thm}\label{theoqfcf}
Assume that $\mathbf{X} \sim \mathcal{E}_{n \times m}^{\beta}(\boldsymbol{0}, \mathbf{\Theta},
\mathbf{\Sigma}, h)$  and define $\mathbf{W} = \mathbf{X}^{*}\mathbf{A}\mathbf{X} \in
\mathfrak{S}_{m}^{\beta}$ of rank $r$, with $\mathbf{A} \in \mathfrak{S}_{n}^{\beta}$ of rank
$r \leq  m \leq n$. Then the characteristic function of $\mathbf{W}$ for a real normed division
algebra, is
\begin{equation}\label{qfcfeq}
    C^{\beta}(m,n)\pi^{\beta mn/2} \sum_{k=0}^{\infty} \sum_{\kappa}\frac{[\beta n/2]_{\kappa}}{\Gamma_{1}^{\beta}[\beta mn/2 + k]\ k!}
    \frac{C_{\kappa}^{\beta}\left(\mathbf{\Theta}\mathbf{A}\right)C_{\kappa}(i \beta \mathbf{\Sigma}\mathbf{S})}
    {C_{\kappa}(\mathbf{I}_{r})} \vartheta
\end{equation}
where $[a]_{\kappa}^{\beta}$ is defined in Theorem \ref{teo2i} and
$$
  \vartheta = \int_{z \in \mathfrak{P}_{1}^{\beta}}f(z) z^{\beta mn/2+k-1} dz
$$
\end{thm}
\begin{proof}
We have that
\begin{eqnarray*}
  \psi_{_{\mathbf{W}}}(\mathbf{S}) &=& \E_{_{\mathbf{W}}}(\etr\{i\mathbf{WS}\}) \\
   &=& \E_{_{\mathbf{X}}}(\etr\{i\mathbf{X}^{*}\mathbf{AXS}\}) \\
   &=& \int_{\mathbf{X} \in \mathfrak{L}^{\beta}_{m,n}}\etr\{i\mathbf{X}^{*}\mathbf{AXS}\}f_{\mathbf{X}}(\mathbf{X})(d\mathbf{X})
\end{eqnarray*}
Considering the transformation $\mathbf{Y} =
\mathbf{\Theta}^{-1/2}\mathbf{X}\mathbf{\Sigma}^{-1/2}$, where $(\mathbf{B}^{1/2})^{2} =
\mathbf{B}$, then by Proposition \ref{lemlt}, $(d\mathbf{X}) =
|\mathbf{\Theta}|^{m/2}|\mathbf{\Sigma}|^{n/2} (d\mathbf{Y})$. Hence from (\ref{v-sd}) the
characteristic function of $\mathbf{W}$ is
\begin{equation}\label{cfeq1}
    C^{\beta}(m,n)\int_{\mathbf{Y} \in \mathfrak{L}^{\beta}_{m,n}}\etr\{i\mathbf{\Sigma}^{1/2}
    \mathbf{Y}^{*}\mathbf{\Theta}^{1/2}\mathbf{A} \mathbf{\Theta}^{1/2}\mathbf{Y}\mathbf{\Sigma}^{1/2}
    \mathbf{S}\}h(\beta \tr \mathbf{Y}^{*}\mathbf{Y})(d\mathbf{Y})
\end{equation}
Now, let $\mathbf{Y} = \mathbf{V}_{1}\mathbf{DW}^{*}$ the singular value decomposition of
$\mathbf{Y}$. Then, by Proposition \ref{lemsvd},
$$
  (d\mathbf{Y}) = 2^{-m}\pi^{\tau} \prod_{i = 1}^{m} d_{i}^{\beta(n - m + 1) -1}
    \prod_{i < j}^{m}(d_{i}^{2} - d_{j}^{2})^{\beta} (d\mathbf{D}) (\mathbf{V}_{1}^{*}d\mathbf{V}_{1})
    (\mathbf{W}^{*}d\mathbf{W}).
$$
Therefore, the integral in (\ref{cfeq1}) is
$$
  2^{-m}\pi^{\tau}\int_{\mathbf{W} \in \mathfrak{U}^{\beta}(m)} \int_{\mathbf{D} \in \mathfrak{D}^{\beta}_{m}}
  \int_{\mathbf{V}_{1} \in \mathcal{V}^{\beta}_{m,n}}
  \etr\{i\mathbf{\Sigma}^{1/2} \mathbf{WDV}_{1}^{*}\mathbf{\Theta}^{1/2}\mathbf{A} \mathbf{\Theta}^{1/2}
  \mathbf{V}_{1}\mathbf{DW}^{*}\mathbf{\Sigma}^{1/2}  \mathbf{S}\}
$$
\begin{equation}\label{cfeq2}
    \times h(\beta \tr \mathbf{WD}^{2}\mathbf{W}^{*})
  \prod_{i = 1}^{m} d_{i}^{\beta(n - m + 1) -1} \prod_{i < j}^{m}(d_{i}^{2} - d_{j}^{2})^{\beta}
  (d\mathbf{D}) (\mathbf{V}_{1}^{*}d\mathbf{V}_{1}) (\mathbf{W}^{*}d\mathbf{W})
\end{equation}
Now, note that the integral on $\mathcal{V}^{\beta}_{m,n}$ in (\ref{cfeq2}) that is
$$
   \int_{\mathbf{V}_{1} \in \mathcal{V}^{\beta}_{m,n}} \etr\{i\mathbf{\Sigma}^{1/2} \mathbf{WDV}_{1}^{*}
   \mathbf{\Theta}^{1/2}\mathbf{A} \mathbf{\Theta}^{1/2} \mathbf{V}_{1}\mathbf{DW}^{*}\mathbf{\Sigma}^{1/2}
   \mathbf{S}\} (\mathbf{V}_{1}^{*}d\mathbf{V}_{1})
$$
can be expressed as
$$
  \sum_{k = 1}^{\infty}\sum_{\kappa}\frac{1}{k!}\int_{\mathbf{V}_{1} \in \mathcal{V}^{\beta}_{m,n}}
  C_{\kappa}^{\beta}\left(i\mathbf{\Sigma}^{1/2} \mathbf{WDV}_{1}^{*}
   \mathbf{\Theta}^{1/2}\mathbf{A} \mathbf{\Theta}^{1/2} \mathbf{V}_{1}\mathbf{DW}^{*}\mathbf{\Sigma}^{1/2}
   \mathbf{S}\right) (\mathbf{V}_{1}^{*}d\mathbf{V}_{1})
$$
which by Theorem \ref{teo1i} is
$$
  \Vol(\mathcal{V}^{\beta}_{m,n}) \sum_{k = 1}^{\infty}\sum_{\kappa}
   \frac{C_{\kappa}^{\beta}\left(\mathbf{A} \mathbf{\Theta}\right)C_{\kappa}^{\beta}\left(i\mathbf{\Sigma}^{1/2}
   \mathbf{S}\mathbf{\Sigma}^{1/2}\mathbf{WD}^{2}\mathbf{W}^{*}\right)}{k!
   C_{\kappa}^{\beta}(\mathbf{I}_{r})}.
$$
Hence, substituting in (\ref{cfeq2}) we obtain
$$
  2^{-m}\pi^{\tau} \Vol(\mathcal{V}^{\beta}_{m,n}) \sum_{k = 1}^{\infty}\sum_{\kappa}
  \frac{C_{\kappa}^{\beta}\left(\mathbf{A} \mathbf{\Theta}\right)}{k!
   C_{\kappa}^{\beta}(\mathbf{I}_{r})}
  \int_{\mathbf{W} \in \mathfrak{U}^{\beta}(m)} \int_{\mathbf{D} \in \mathfrak{D}^{\beta}_{m}}
  C_{\kappa}^{\beta}\left(i\mathbf{\Sigma}^{1/2}
   \mathbf{S}\mathbf{\Sigma}^{1/2}\mathbf{WD}^{2}\mathbf{W}^{*}\right)
$$
\begin{equation}\label{cfeq3}
    \times h(\beta \tr \mathbf{WD}^{2}\mathbf{W}^{*})
  \prod_{i = 1}^{m} d_{i}^{\beta(n - m + 1) -1} \prod_{i < j}^{m}(d_{i}^{2} - d_{j}^{2})^{\beta}
  (d\mathbf{D}) (\mathbf{W}^{*}d\mathbf{W}).
\end{equation}
Now, considering the transformation $\mathbf{V} = \mathbf{WD}^{2}\mathbf{W}^{*}$, then by
Proposition \ref{lemsd}, noting that $\mathbf{\Lambda} = \mathbf{D}^{2}$, and
$(d\mathbf{\Lambda}) = 2^{m}|\mathbf{D}|(d\mathbf{D})$,
$$
    (d\mathbf{V}) = \pi^{\tau} \prod_{i < j}^{m} (d_{i}^{2} - d_{j}^{2})^{\beta}
    \prod_{i=1}^{m}d_{i}(d\mathbf{D})(\mathbf{W}^{*}d\mathbf{W}),
$$
Thus, by (\ref{cfeq1}) and (\ref{cfeq3}) the characteristic function of $\mathbf{W}$ is
$$
  C^{\beta}(m,n)2^{-m} \Vol(\mathcal{V}^{\beta}_{m,n}) \sum_{k = 1}^{\infty}\sum_{\kappa}
  \frac{C_{\kappa}^{\beta}\left(\mathbf{A} \mathbf{\Theta}\right)}{k!
   C_{\kappa}^{\beta}(\mathbf{I}_{r})}
  \int_{\mathbf{V} \in \mathfrak{P}^{\beta}_{m}}
  C_{\kappa}^{\beta}\left(i\mathbf{\Sigma}^{1/2}
   \mathbf{S}\mathbf{\Sigma}^{1/2}\mathbf{V}\right)
$$
\begin{equation}\label{cfeq4}
   \hspace{6.5cm}\times \ h(\beta \tr \mathbf{V}) |\mathbf{V}|^{\beta(n - m + 1)/2 - 1}
  (d\mathbf{V}).
\end{equation}
Finally the desired result is obtained integrating (\ref{cfeq4}) using Theorem \ref{teo2i}.
\qed
\end{proof}

\begin{cor}\label{corqfcfN}
 Assume that $\mathbf{X}$ has a matrix multivariate normal distribution for real normed
division algebras,  and define $\mathbf{W} = \mathbf{X}^{*}\mathbf{A}\mathbf{X} \in
\mathfrak{S}_{m}^{\beta}$ of rank $r$, with $\mathbf{A} \in \mathfrak{S}_{n}^{\beta}$ of rank
$r \leq  m \leq n$. Then the characteristic function of $\mathbf{W}$ for a real normed division
algebra, is
\begin{equation}\label{qfcfNeq}
    2^{\beta-1}\beta^{\beta mn/2} \sum_{k=0}^{\infty} \sum_{\kappa}\frac{[\beta n/2]_{\kappa}}{\ k!}
    \frac{C_{\kappa}^{\beta}\left(\mathbf{\Theta}\mathbf{A}\right)C_{\kappa}(2i\beta \mathbf{\Sigma}\mathbf{S})}
    {C_{\kappa}(\mathbf{I}_{r})}.
\end{equation}
\end{cor}
\begin{proof}
Observe that, for the normal case, $h(u) = \exp\{- u/2\}$ and $C^{\beta}(m,n) = (2
\pi/\beta)^{-\beta mn/2}$. Then in this case
\begin{eqnarray*}
  \vartheta &=& \int_{z \in \mathfrak{P}_{1}^{\beta}}f(z) z^{\beta mn/2+k-1} dz \\
   &=& \int_{z \in \mathfrak{P}_{1}^{\beta}}\exp\{- z/2\} z^{\beta mn/2+k-1} dz =
   2^{\beta mn/2+k-1+\beta} \Gamma_{1}^{\beta}[\beta mn/2+k].
\end{eqnarray*}
Then the result follows.\qed
\end{proof}
In the real case, i.e. $\beta = 1$, this result was obtained by \citet{k:66} in terms of zonal
polynomials y by \citet{s:70} in terms of Laguerre polynomials with matrix argument when
$\mathbf{A}$ is a positive definite matrix.

\begin{cor}\label{corqfcfT}
 Assume that $\mathbf{X}$ has a matrix multivariate $t$ distribution for real normed
division algebras,  and define $\mathbf{W} = \mathbf{X}^{*}\mathbf{A}\mathbf{X} \in
\mathfrak{S}_{m}^{\beta}$ of rank $r$, with $\mathbf{A} \in \mathfrak{S}_{n}^{\beta}$ of rank
$r \leq  m \leq n$. Then the characteristic function of $\mathbf{W}$ for a real normed division
algebra, is
\begin{equation}\label{qfcfTeq}
    g^{\beta-1}\beta^{\beta mn/2}
    \sum_{k=0}^{\infty} \sum_{\kappa}\frac{[\beta n/2]_{\kappa}}{ k! \ (s-\beta mn/2-k)_{k}}
    \frac{C_{\kappa}^{\beta}\left(\mathbf{\Theta}\mathbf{A}\right)C_{\kappa}(i g\beta \mathbf{\Sigma}\mathbf{S})}
    {C_{\kappa}(\mathbf{I}_{r})},
\end{equation}
where $s, g \in \Re$, $s, g > 0$, $s > \beta mn/2$.
\end{cor}
\begin{proof}
In this case we have
$$
  C^{\beta}(m,n)=\frac{\Gamma_{1}^{\beta}[s]}{(\pi g \beta^{-1})^{\beta mn/2}\Gamma_{1}^{\beta}\left[s-\beta mn/2\right]},
  \quad \mbox{ and } \quad h(u)= \left(1+u/g\right)^{-s},
$$
with $s, g \in \Re$, $s, g > 0$, $s > \beta mn/2$. Then
\begin{eqnarray*}
  \vartheta &=& \int_{z \in \mathfrak{P}_{1}^{\beta}}f(z) z^{\beta mn/2+k-1} dz \\
   &=& \int_{z \in \mathfrak{P}_{1}^{\beta}}\left(1+z/g\right)^{-s} z^{\beta mn/2+k-1} dz \\
   &=& g^{\beta mn/2+k-1+\beta} \frac{\Gamma_{1}^{\beta}[\beta mn/2+k] \Gamma_{1}^{\beta}
   [s-\beta mn/2-k]}{\Gamma_{1}^{\beta}[s]}.
\end{eqnarray*}
Finally, recalling that $\Gamma_{1}^{\beta}[a] = \Gamma_{1}^{\beta}[a+k]/(a)_{k}$ for $a+k>0$,
$k$ an integer number. Then, taking $a = s-\beta mn/2-k$
$$
  \Gamma_{1}^{\beta}[s-\beta mn/2-k] = \frac{\Gamma_{1}^{\beta}[s-\beta mn/2]}{(s-\beta
  mn/2-k)_{k}},
$$
the desired result is obtained.\qed
\end{proof}

In corollaries \ref{cor:P7} and \ref{corqfcfT} observe that, when $s =(\beta mn+g)/2$,
$\mathbf{W}$ is said to have a matrix quadratic form $t$ distribution for real normed division
algebras with $g$ degrees of freedom. And in this case, if $g =1$, then $\mathbf{W}$ is said to
have a matrix quadratic form Cauchy distribution for real normed division algebras.

\section*{Conclusions}

Any reader interested in a particular case -- real, complex, quaternions or octonions -- need
simply take the particular value of $\beta$ in order to obtain the results desired.
Furthermore, as is asseverated by \citet{k:84}, our results can be extended to hypercomplex
cases, by simply substituting $\beta$ by $2\beta$, obtaining the complex, bicomplex,
biquaternion and bioctonion (or sedenionic) cases. Observe that, alternatively to use the
concepts and notation of real normed division algebras, it is possible to use the concepts and
notation of simple Jordan algebras noting that one particular algebra more is contained. As was
mentioned, there are four real normed division algebras and five Euclidean simple Jordan
algebras, see \citet{cl:96}.

In summary, the density and characteristics functions of a matrix quadratics forms of matrix
multivariate elliptical distribution are found under a unified approach that allows the
simultaneous study of the real, complex, quaternion and octonion cases, generically termed
distributions for real normed division algebras. In particular these results were
particularised for matrix quadratics forms of a matrix multivariate normal, Pearson type VII,
$t$ and Cauchy distributions for real normed division algebras.

Finally, note that with particular cases of the results obtained in this paper, one can obtain
many of the results published in the literature. Thus for example, if in Corollary
\ref{corqfcfN}, $\mathbf{A}$ is a idempotent matrix and $\mathbf{\Theta} = \mathbf{I}$, we
obtained
\begin{eqnarray*}
  \psi_{_{\mathbf{W}}}(\mathbf{S}) &=& 2^{\beta-1}\beta^{\beta mn/2} \sum_{k=0}^{\infty} \sum_{\kappa}\frac{[\beta n/2]_{\kappa}}{\ k!}
    \frac{C_{\kappa}^{\beta}\left(\mathbf{A}\right)C_{\kappa}(2i\beta \mathbf{\Sigma}\mathbf{S})}
    {C_{\kappa}(\mathbf{I}_{r})}. \\
   &=& 2^{\beta-1}\beta^{\beta mn/2}|\mathbf{I}- 2i\beta \mathbf{\Sigma}\mathbf{S})|^{-\beta
   n/2},
\end{eqnarray*}
and for real case $\psi_{_{\mathbf{W}}}(\mathbf{S}) = |\mathbf{I}-
2i\mathbf{\Sigma}\mathbf{S})|^{- n/2}$.



\end{document}